\documentclass[12pt]{amsart}

\pdfoutput=1

\usepackage{microtype}

\usepackage{verbatim}
\usepackage{amssymb}
\usepackage{upref}
\usepackage[all]{xy}
\usepackage[usenames,dvipsnames]{color}
\usepackage{url}
\usepackage[normalem]{ulem}

\allowdisplaybreaks

\newtheorem{thm}{Theorem}[section]
\newtheorem*{thm*}{Theorem}
\newtheorem{lem}[thm]{Lemma}
\newtheorem{cor}[thm]{Corollary}
\newtheorem{prop}[thm]{Proposition}

\theoremstyle{definition}
\newtheorem{defn}[thm]{Definition}

\newtheorem*{notn*}{Notation}

\newtheorem*{hyp*}{Hypothesis}

\newtheorem{rem}[thm]{Remark}
\newtheorem*{rem*}{Remark}

\numberwithin{equation}{section}

\newcommand{\secref}[1]{Section~\textup{\ref{#1}}}

\newcommand{\thmref}[1]{Theorem~\textup{\ref{#1}}}

\newcommand{\propref}[1]{Proposition~\textup{\ref{#1}}}

\newcommand{\midtext}[1]{\quad\text{#1}\quad}
\newcommand{\righttext}[1]{\quad\text{#1 }}
\renewcommand{\and}{\midtext{and}}

\renewcommand{\)}{\textup)}

\newcommand{\C}{\mathbb C}

\newcommand{\CC}{\mathcal C}
\newcommand{\DD}{\mathcal D}
\newcommand{\KK}{\mathcal K}

\newcommand{\LL}{\mathcal L}

\renewcommand{\epsilon}{\varepsilon}

\DeclareMathOperator{\ad}{Ad}

\DeclareMathOperator{\mor}{Mor}

\DeclareMathOperator*{\spn}{span}
\DeclareMathOperator*{\clspn}{\overline{\spn}}

\newcommand{\id}{\text{\textup{id}}}

\renewcommand{\>}{\rangle}

\newcommand{\inv}{^{-1}}

\newcommand{\variso}{\overset{\simeq}{\longrightarrow}}

\newcommand{\what}{\widehat}
\newcommand{\wilde}{\widetilde}

\newcommand{\smtx}[1]{\left(\begin{smallmatrix} #1
\end{smallmatrix}\right)}
\newcommand{\mtx}[1]{\begin{pmatrix} #1 \end{pmatrix}}

\newcommand{\sscs}{\mathbf{SS}\!\text{ - }\!\cs}
\newcommand{\ssfcsen}{\mathbf{SSF}\!\text{ - }\!\cse}

\newcommand{\cs}{\mathbf{C}^*}
\newcommand{\csn}{\cs\nd}
\newcommand{\cse}{\cs\en}

\newcommand{\kalgn}{\KK\!\text{ - }\!\csn}
\newcommand{\kalge}{\KK\!\text{ - }\!\cse}

\newcommand{\nd}{_\mathbf{nd}}
\newcommand{\en}{_\mathbf{en}}

\newcommand{\st}{\textup{St}}
\newcommand{\stn}{\st\nd}
\newcommand{\ste}{\st\en}

\newcommand{\dst}{\textup{DSt}}
\newcommand{\dstn}{\dst\nd}

\newcommand{\wst}{\wilde{\st}}
\newcommand{\wstn}{\wst\nd}
\newcommand{\wste}{\wst\en}

\newcommand{\dn}{\downarrow}

\begin{document}
\title{Destabilization}

\author[Kaliszewski]{S.~Kaliszewski}
\address{School of Mathematical and Statistical Sciences
\\Arizona State University
\\Tempe, Arizona 85287}
\email{kaliszewski@asu.edu}
\author[Omland]{Tron Omland}
\address{School of Mathematical and Statistical Sciences
\\Arizona State University
\\Tempe, Arizona 85287}
\email{omland@asu.edu}
\author[Quigg]{John Quigg}
\address{School of Mathematical and Statistical Sciences
\\Arizona State University
\\Tempe, Arizona 85287}
\email{quigg@asu.edu}



\date{August 13, 2015}

\begin{abstract}
This partly expository paper first supplies the details of a
method of factoring a stable $C^*$-algebra $A$ as $B\otimes \KK$ in a canonical way.
Then it is shown that this method can be put into a categorical framework, much like the crossed-product dualities,
and that stabilization gives rise to an equivalence between the nondegenerate category of $C^*$-algebras and a category of ``$\KK$-algebras''.
We consider this equivalence as ``inverting''
the stabilization process,
that is, a ``destabilization''.

Furthermore, the method of factoring stable $C^*$-algebras generalizes to Hilbert bimodules,
and an analogous category equivalence between the associated enchilada categories is produced,
giving a destabilization for $C^*$-correspondences.

Finally, we make a connection with (double) crossed-product duality.
\end{abstract}
\maketitle

\section{Introduction}
\label{intro}

We start with some
well-known facts
that nowadays can be found in any textbook on $C^*$-algebras.
The \emph{stabilization} of a $C^*$-algebra $A$ is $A\otimes\KK$, where $\KK$ denotes the compact operators on a separable infinite-dimensional Hilbert space.
Since $\KK\otimes\KK\simeq\KK$,
stabilizing twice does not (up to isomorphism) produce anything new, as the name suggests.
In other words, stabilizations of $C^*$-algebras, are precisely those $C^*$-algebras $A$ such that $A\simeq A\otimes\KK$,
and any $A$ satisfying this property is simply called \emph{stable}.
An obvious question is: how can we characterize stable $C^*$-algebras?
While this problem may have several approaches, our goal is to answer the following:
given a $C^*$-algebra $A$, how can we decide whether there exists \emph{some} $C^*$-algebra $B$ such that $A\simeq B\otimes\KK$,
and then how can we produce such a $B$ in a canonical way?
A trivial answer comes straight from the definition,
namely
there exists such a $B$ if and only if $A\simeq A\otimes\KK$, in which case we can take $B=A$.
This is unsatisfying on two levels.
First of all, the
property
$A\simeq A\otimes\KK$ is not very convenient to check,
and secondly, the choice $B=A$ does not allow for a possibly unstable algebra $B$.
Another,
possibly more useful, characterization of stable $C^*$-algebras that
seems to us to be
folklore
is that
$A$ is stable if and only if there is a nondegenerate copy of $\KK$ in $M(A)$.
Then one can use a choice of matrix units in $\KK$ to decompose $A$ as 
infinite matrices
whose entries come from a $C^*$-algebra $B$, and then $A\simeq B\otimes\KK$.
The challenge is
to produce a $C^*$-algebra that is isomorphic to this $B$ without having to choose matrix units, i.e., canonically.
One way becomes apparent by considering how to pick $B$ out of $B\otimes\KK$.
We have an injection 
from $B$ to $M(B\otimes\KK)$ given by $b\mapsto b\otimes 1_\KK$,
where $1_\KK$ denotes the identity element of $M(\KK)$.
The trick is to identify the image $B\otimes 1_\KK$ inside $M(B\otimes\KK)$.
Obviously $B\otimes 1_\KK$ commutes with $1_B \otimes\KK$, and also multiplies $1_B\otimes\KK$ into $B\otimes\KK$.
This gives the characterization of interest to us, and a short, very rough, summary (using different arguments from those we present here) can be found in
\cite[Section~3]{fischer}.
We feel that it is useful to ``officially'' record the details for convenient reference, 
since it seems difficult to dig them out of the literature,
and moreover we think it is appropriate to make our arguments as elementary as possible.
We give the details in \propref{AK}, after recalling some background material in \secref{prelim}.

In \thmref{destabilize} we parlay the characterization of \propref{AK} into an equivalence between the categories of $C^*$-algebras and of ``$\KK$-algebras''
(stable $C^*$-algebras equipped with a given embedding of $\KK$ --- see \secref{kalg} for the definitions).
Here the morphisms in both categories involve nondegenerate homomorphisms into multiplier algebras.
We then discuss how this category equivalence fits into a general framework we described in \cite[Section~4]{koqlandstad}:
if we consider the stabilization process $B\mapsto A=B\otimes\KK$,
\thmref{destabilize} tells us what extra data we need to recover $B$ from $A$,
i.e., to ``invert the process''.
In \cite{koqlandstad} we study this in a technical manner,
using basic category theory,
and we introduce a concept we call ``good inversion'', where the inverse image of an output of the process is classified up to isomorphism by the automorphisms of the output.
We observe in \secref{equivalence} that \thmref{destabilize} is an example of a good inversion.

In \secref{bimodule} we extend \propref{AK} to Hilbert bimodules,
and in \secref{enchilada} we apply this to make stabilization into an equivalence between enchilada%
\footnote{the term ``enchilada'' is informal, and originated when the authors of the AMS Memoir \cite{enchilada} published a smaller paper \cite{taco} as an introduction to the techniques, and referred to the smaller paper as the ``little taco'' and the memoir as the ``big enchilada''}
categories,
where the morphisms come from $C^*$-correspondences.

Recall that two $C^*$-algebras $A$ and $B$ are \emph{stably isomorphic} if $A\otimes\KK\simeq B\otimes\KK$,
and this condition is stronger than Morita equivalence,
that is, \emph{every} $C^*$-algebra $A$ is stable up to Morita equivalence.
Therefore, one should think of the results of \secref{bimodule} and \secref{enchilada}
as a characterization of which 
$C^*$-correspondences
are stable.

Furthermore, it is a consequence of the Brown-Green-Rieffel theorem \cite{bgr}
that separable $C^*$-algebras are Morita equivalent if and only if they are stably isomorphic.
In light of this,
we briefly compare the nondegenerate and enchilada categories
when restricted to the subcategories of separable stable $C^*$-algebras.

Finally, in an another extended remark at the end
we make a connection with duality for the double crossed product by an action or a coaction.

\section{Preliminaries}\label{prelim}

The category equivalences we will develop in this paper are based upon two categories of $C^*$-algebras,
the nondegenerate and the enchilada.

\begin{defn}
In the \emph{nondegenerate category $\csn$ of $C^*$-algebras},
a morphism $\phi\colon A\to B$
is a nondegenerate homomorphism $\phi\colon A\to M(B)$.
\end{defn}

Before defining the enchilada category, we recall some basic facts about $C^*$-correspondences.

\subsection{$C^*$-Correspondences}

We refer to \cite[Chapters~1--2]{enchilada} for $C^*$-correspondences and their multipliers.
Given $C^*$-algebras $A$ and $B$, recall that an \emph{$A-B$ correspondence} is a (right) Hilbert $B$-module $X$ equipped with a homomorphism $\varphi\colon A\to \LL(X)$.
We frequently write $a\cdot x=\varphi(a)x$.
Further, $X$ is called \emph{nondegenerate} if $A\cdot X=X$,
and \emph{full} if $\clspn\<X,X\>_B=B$.
We will always require our correspondences to be nondegenerate, but we do not need them to be full, i.e., $\clspn\<X,X\>_B$ might be a proper ideal of $B$.
In \cite{enchilada}, nondegenerate $A-B$ correspondence are called \emph{right-Hilbert $A-B$ bimodules}.

A $C^*$-algebra $D$ can be regarded as a correspondence over itself in the standard manner, using the algebraic operations of $D$,
and then if $X$ is a nondegenerate $A-B$ correspondence
the \emph{external tensor product} $X\otimes D$ is a nondegenerate $(A\otimes D)-(B\otimes D)$ correspondence.

An \emph{$A-B$ Hilbert bimodule} is an $A-B$ correspondence $X$ that also is a left Hilbert $A$-module and satisfies the compatibility condition
${}_A\<x,y\>\cdot z=x\cdot \<y,z\>_B$;
\cite{enchilada} would call $X$ a \emph{partial imprimitivity bimodule}.
Note that, just as a Hilbert $B$-module is automatically nondegenerate as a right $B$-module,
an $A-B$ Hilbert bimodule is automatically nondegenerate as an $A-B$ correspondence.
We call an $A-B$ Hilbert bimodule \emph{right-full} if it is full as a Hilbert $B$-module,
and \emph{left-full} if it is full as a left Hilbert $A$-module;
\cite{enchilada} uses the terms \emph{left-partial imprimitivity bimodule}
and \emph{right-partial imprimitivity bimodule}, respectively.
Of course, we call a Hilbert bimodule that is both left- and right-full an \emph{imprimitivity bimodule}.
The \emph{linking algebra} of an $A-B$ Hilbert bimodule $X$ is $L(X)=\smtx{A&X\\\wilde X&B}$.

The \emph{multiplier bimodule} of an $A-B$ correspondence is $M(X)=\LL_B(B,X)$ (see \cite[Definition~1.14]{enchilada}), which is an $M(A)-M(B)$ correspondence in a natural way (see \cite[Proposition~1.10]{enchilada}).

\begin{defn}[{\cite[Chapter~2]{enchilada}}]
The \emph{enchilada category $\cse$ of $C^*$-algebras}
has the same objects as $\csn$,
but now when we say $[X]\colon A\to B$ is a morphism in the category we mean $[X]$ is the isomorphism class of a nondegenerate $A-B$ correspondence $X$.
Composition of morphisms is given by balanced tensor products, and the identity morphisms by the $C^*$-algebras themselves,
viewed as correspondences in the standard way.
\end{defn}

\subsection{Category equivalences and inversions}
    
Recall the following elementary concepts from category theory:
a functor $F\colon \CC\to\DD$ between categories $\CC$ and $\DD$ is
\begin{itemize}
\item
\emph{full} if $F$ maps $\mor(x,y)$ surjectively to $\mor(Fx,Fy)$ for all objects $x,y$ in $\CC$;

\item
\emph{faithful} if $F$ maps $\mor(x,y)$ injectively to $\mor(Fx,Fy)$ for all objects $x,y$ in $\CC$;

\item
\emph{essentially surjective} if every object in $\DD$ is isomorphic to one in the image of $F$;

\item
an \emph{equivalence} if it has a \emph{quasi-inverse}, that is, a functor $G\colon \DD\to\CC$ such that the compositions $GF$ and $FG$ are naturally isomorphic to the identity functors.
\end{itemize}
We need the following result from \cite[Section~IV.4, Theorem~1]{maclane}:
$F\colon \CC\to\DD$ is an equivalence if and only if it is full, faithful, and essentially surjective.
Moreover, it follows from the proof of the theorem in \cite{maclane} that
if for every object $y$ of $\DD$ we choose an object $Gy$ of $\CC$ and an isomorphism
$\theta_y\colon FGy\to y$,
then $G$ extends uniquely to a quasi-inverse of $F$ whose morphism map
has the following universal property:
for every morphism $g\colon y\to z$ in $\DD$,
$Gg$ is the unique morphism in $\CC$ making the diagram
\[
\xymatrix{
Gy \ar@{-->}[d]^{!}_{Gg}
&FGy \ar[r]^-{\theta_y}_-\simeq \ar[d]_{FGg}
&y \ar[d]^g
\\
Gz
&FGz \ar[r]_-{\theta_z}^-\simeq
&z
}
\]
commute.
When we need to refer to this, we will just say ``by generalized abstract nonsense''.

A 
morphism $\phi$ is an isomorphism in the nondegenerate category if and only if it is a $C^*$-isomorphism in the usual sense,
and a
morphism $[X]$ is an isomorphism in the enchilada category if and only if $X$ is an imprimitivity bimodule.

Recall from \cite[Definition~4.1]{koqlandstad} that
we say
an \emph{inversion} of a functor $P\colon \CC\to \DD$
is a commutative diagram
\[
\xymatrix{
\CC \ar[r]^{\wilde P} \ar[dr]_P
&\wilde \DD \ar[d]^F
\\
&\DD
}
\]
of functors such that
\begin{itemize}
\item[(i)] $\wilde P$ is an equivalence of categories;
\item[(ii)] $\wilde \DD$ is a category whose objects are pairs $(A,\sigma)$,
where $A$ is an object of $\DD$ and $\sigma$ denotes some extra structure;
\item[(iii)] $F$ is defined by $F(A,\sigma)=A$ on objects, and is faithful.
\end{itemize}
Moreover, the inversion is \emph{good}
if both
\begin{enumerate}
\item
the image of $F$ is contained in the essential image of $P$,
and 

\item
$F$ has the following
\emph{unique isomorphism lifting property}:
whenever $y\in \DD$ and $u\in F\inv(y)$,
for every isomorphism $\theta$ in $\DD$ with domain $y$
there is a unique isomorphism $\theta_u$ in $\wilde \DD$ with domain $u$
such that $F(\theta_u)=\theta$.
\end{enumerate}

\section{$\KK$-algebras}\label{kalg}

We record here, for convenient reference,
a folklore result (\propref{AK}) that seems surprisingly difficult to dig out of the literature.
We emphasize that nothing in this section is new.
Let $\KK$ denote the compact operators on a separable infinite-dimensional Hilbert space.

\begin{defn}
A \emph{$\KK$-algebra} is a pair $(A,\iota)$, where $A$ is a $C^*$-algebra and $\iota\colon \KK\to M(A)$ is a nondegenerate homomorphism.
\end{defn}

Let $\{u_{ij}\}$ be a system of matrix units for $\KK$.
Then 
$e_{ij}=\iota(u_{ij})$
gives
a system 
of matrix units in $M(A)$, and
\[
\sum_ie_{ii}=1\righttext{strictly in $M(A)$}.
\]
Conversely, every such system of matrix units in $M(A)$ uniquely determines a nondegenerate homomorphism $\iota\colon \KK\to M(A)$.

\begin{defn}
In the above situation we call $\{e_{ij}\}_{i,j=1}^\infty$ a \emph{system of matrix units associated to $\iota$.}
\end{defn}

\begin{defn}
Let $(A,\iota)$ be a $\KK$-algebra.
The \emph{$A$-relative commutant} of $\KK$ associated to $\iota$ is
\[
C(A,\iota):=\{a\in M(A):\iota(k)a=a\iota(k)\in A\text{ for all }k\in\KK\}.
\]
\end{defn}

\begin{prop}\label{AK}
If $(A,\iota)$ is a $\KK$-algebra,
then
the $A$-relative commutant $C(A,\iota)$
is a nondegenerate $C^*$-subalgebra of $M(A)$, and
there is an isomorphism
\[
\theta
\colon C(A,\iota)\otimes\KK\variso A
\]
given on elementary tensors by
\[
\theta(a\otimes k)=a\iota(k).
\]

Furthermore, if $B$ is a $C^*$-algebra,
and if $B\otimes\KK$ is
regarded as a $\KK$-algebra via the map
\[
1\otimes\id\colon k\mapsto 1_B\otimes k
\]
then
the map
\[
\id\otimes 1\colon b\mapsto b\otimes 1
\]
gives an isomorphism
\[
B\simeq C(B\otimes\KK,1\otimes\id).
\]
\end{prop}

\begin{proof}
Let $\{e_{ij}\}$ be a system of matrix units associated to $\iota$.
For each $i,j$ put
\[
A_{ij}=e_{ii}Ae_{jj}.
\]
Then
$A$ is the inductive limit of the increasing family of $C^*$-subalgebras
\[
A|_n:=
\sum_{i,j=1}^nA_{ij},
\]
and the subspaces $A_{ij}$ are linearly independent.
For each $i,j$ there is an isometric linear bijection $\tau_{ij}\colon A_{ij}\to A_{11}$ given by
\[
\tau_{ij}(a)=e_{1i}ae_{j1},
\]
with $\tau_{ij}\inv\colon A_{11}\to A_{ij}$ given by
$a\mapsto e_{i1}ae_{1j}$.
Then there is a $C^*$-isomorphism $\varphi_n\colon A|_n\to A_{11}\otimes M_n$ given by
\[
\varphi_n\left(\sum_{i,j=1}^na_{ij}\right)=\sum_{i,j=1}^n\bigl(\tau_{ij}(a_{ij})\otimes u_{ij}\bigr).
\]
For each $n$,
identify $M_n$ with the upper left $n\times n$ corner of $M_{n+1}$.
We have commuting diagrams
\[
\xymatrix@C+20pt{
A|_n \ar[r]^-{\varphi_n}_-\simeq \ar@{^(->}[d]
&A_{11}\otimes M_n \ar@{^(->}[d]
\\
A|_{n+1} \ar[r]_-{\varphi_{n+1}}^-\simeq
&A_{11}\otimes M_{n+1},
}
\]
and taking the inductive limit gives an isomorphism
\begin{equation*}
\varphi=\varinjlim \varphi_n\colon \varinjlim A|_n=A\variso \varinjlim (A_{11}\otimes M_n)=A_{11}\otimes \KK.
\end{equation*}

We now show that
$C(A,\iota)$ is a nondegenerate $C^*$-subalgebra of $M(A)$,
and that
there is an isomorphism $\psi\colon C(A,\iota)\to A_{11}$ given by
\[
\psi(a)=ae_{11}.
\]
It is clear that $C(A,\iota)$ is a $C^*$-subalgebra of $M(A)$, and
$\psi$ is a 
homomorphism
because $e_{11}$ commutes with $C(A,\iota)$.

Since the partial isometries $\{e_{i1}\}_{i=1}^\infty$ have pairwise orthogonal range projections
that sum strictly to 1 in $M(A)$,
we can define a nondegenerate 
homomorphism $\sigma\colon A_{11}\to M(A)$ by
\[
\sigma(d)=\sum_i \ad e_{i1}(d)=\sum_i \tau_{ii}\inv(d),
\]
where the sum converges strictly.
The following computations show that $\sigma(A_{11})\subset C(A,\iota)$:
\begin{align*}
e_{ij}\sigma(d)
&=\sum_ne_{ij}e_{n1}de_{1n}
\\&=e_{i1}de_{1j}\righttext{(which is in $A$)}
\\&=\sum_ne_{n1}de_{1n}e_{ij}
\\&=\sigma(d)e_{ij},
\end{align*}
that
$\sigma$ is a left inverse for $\psi$:
\begin{align*}
\sigma\circ\psi(a)
&=\sum_ie_{i1}ae_{11}e_{1i}
\\&=\sum_iae_{i1}e_{1i}
\\&=\sum_iae_{ii}
\\&=a,
\end{align*}
and is also a right inverse:
\[
\psi\circ\sigma(d)=\sum_ie_{i1}de_{1i}e_{11}=e_{11}de_{11}=d.
\]
We see that $C(A,\iota)$ is nondegenerate in $M(A)$ since $\varphi(A_{11})$ is.

We will now show that the isomorphism
\[
\theta:=
\varphi\inv\circ(\psi\otimes\id)\colon C(A,\iota)\otimes \KK\variso A
\]
is given on elementary tensors by
\[
\theta(a\otimes k)=a\iota(k),
\]
and it suffices to check on 
matrix units:
\begin{align*}
\theta(a\otimes u_{ij})
&=\varphi\inv(ae_{11}\otimes u_{ij})
\\&=\varphi_n\inv(ae_{11}\otimes u_{ij})
\righttext{for any $n\ge i,j$}
\\&=\tau_{ij}\inv(ae_{11})
\\&=e_{i1}ae_{11}e_{1j}
\\&=ae_{i1}e_{1j}
\\&=ae_{ij}
\\&=a\iota(u_{ij}).
\end{align*}

For the other part,
let $B$ be a $C^*$-algebra, and let $\{u_{ij}\}$ be a system of matrix units for $\KK$.
Then
\[
e_{ij}=1_B\otimes u_{ij}
\]
gives a nondegenerate system of matrix units in $M(B\otimes\KK)$, and
in the notation of the above construction we have
\begin{align*}
(B\otimes\KK)_{ij}&=B\otimes u_{ij}&&\righttext{(since $u_{ii}\KK u_{jj}=\C u_{ij}$)}\\
\tau_{ii}\inv(b\otimes u_{11})&=b\otimes u_{ii}\\
\sigma(b)&=b\otimes 1_{\KK}\\
C(B\otimes\KK,1\otimes\id)&=B\otimes 1_\KK\simeq B.
&&\qedhere
\end{align*}
\end{proof}

It follows from \propref{AK} that a $C^*$-algebra is stable if and only if its multiplier algebra contains a nondegenerate copy of $\KK$;
this characterization
is contained in \cite[Theorem~2.1 (c)]{HjeRorStability},
although the version we have stated is folklore that surely has been around much longer.

\begin{rem}
If $(A,\iota)$ is a $\KK$-algebra, then
\[
\|a\iota(k)\|=\|a\|\|k\|\righttext{for all}a\in C(A,\iota),k\in\KK.
\]
Indeed,
this follows immediately from the isomorphism $a\otimes k\mapsto a\iota(k)$ and the equality $\|a\otimes k\|=\|a\|\|k\|$.
\end{rem}

For later convenience we identify the multipliers of the relative commutant of $\KK$:

\begin{lem}\label{multiplier}
If $(A,\iota)$ is a $\KK$-algebra, then
\[
M(C(A,\iota))=\{a\in M(A):\iota(k)a=a\iota(k)\text{ for all }k\in\KK\}.
\]
In particular, if $B$ is a $C^*$-algebra then
\[
M\bigl(C(B\otimes\KK,1\otimes\id)\bigr)=M(B)\otimes 1_{\KK}\subset M(B\otimes\KK).
\]
\end{lem}

\begin{proof}
First let $a\in M(C(A,\iota))$.
We need to show that 
$a$ commutes with every $e_{ij}$.
For all $d\in A_{11}$ we have
\begin{align*}
e_{ij}a\sigma(d)
&=a\sigma(d)e_{ij}\righttext{(because $a\sigma(d)\in C(A,\iota)$)}
\\&=a\sum_n e_{n1}de_{1n}e_{ij}
\\&=ae_{i1}de_{1j}
\\&=a\sum_ne_{ij}e_{n1}de_{1n}
\\&=ae_{ij}\sigma(d),
\end{align*}
which suffices because $\sigma(A_{11})$ is nondegenerate in $M(A)$.

On the other hand, suppose $a$ commutes with every $e_{ij}$.
Then for all $d\in A_{11}$,
\begin{align*}
\sigma(d)a
&=\sum_ie_{i1}de_{1i}a
\\&=\sum_ie_{i1}dae_{1i}
\\&=\sigma(da)\righttext{(because $da\in A_{11}$)},
\end{align*}
and hence $C(A,\iota) a\subset C(A,\iota)$.
Similarly, $aC(A,\iota)\subset C(A,\iota)$.
Thus $a$ idealizes $C(A,\iota)$,
and so $a\in M(C(A,\iota))$ since $C(A,\iota)$ is nondegenerate in $M(A)$.
\end{proof}

\section{Nondegenerate category equivalence}\label{equivalence}

In this section we parlay $A\mapsto C(A,\iota)$ into a quasi-inverse for the stabilization functor $B\mapsto B\otimes\KK$.

\begin{defn}
In the \emph{nondegenerate category $\kalgn$ of $\KK$-algebras} $(A,\iota)$,
a morphism $\phi\colon (A,\iota)\to (B,\zeta)$
is a nondegenerate homomorphism $\phi\colon A\to M(B)$
such that
$\phi\circ\iota=\zeta$.
\end{defn}

\begin{rem}
In categorical terminology, $\kalgn$ is precisely the
coslice category 
of objects in $\csn$ under $\KK$,
often denoted
$\KK\dn\csn$
or $\KK/\csn$,
and is a special type of
comma category.
\end{rem}

\begin{defn}
The 
\emph{nondegenerate stabilization functor} $\stn\colon \csn\to \csn$
is given on objects by $B\mapsto B\otimes\KK$
and on morphisms by
$\phi\mapsto \phi\otimes\id_\KK$,
and the
\emph{embellished nondegenerate stabilization functor} $\wstn\colon \csn\to \kalgn$
is given on objects by $B\mapsto (B\otimes\KK,1\otimes\id_\KK)$
and on morphisms by
$\phi\mapsto \phi\otimes\id_\KK$.
\end{defn}
The above is justified by the well-known facts that
if $\phi\colon A\to M(B)$ is a nondegenerate homomorphism
then 
\[
\phi\otimes\id_\KK\colon A\otimes\KK\to M(B\otimes\KK)
\]
is a nondegenerate homomorphism such that
\[
(\phi\otimes\id_\KK)(1_A\otimes k)=1_B\otimes k\righttext{for all}k\in\KK,
\]
and 
that $\phi\mapsto \phi\otimes\id_\KK$ is functorial.

\begin{thm}\label{destabilize}
The embellished nondegenerate stabilization functor $\wstn\colon \csn\to\kalgn$ is a category equivalence,
and has a 
unique 
quasi-inverse functor 
whose object map is
$(A,\iota)\mapsto C(A,\iota)$.

Moreover, this quasi-inverse takes a morphism $\psi\colon (A,\iota)\to (B,\zeta)$ to
the ``restriction'' $\psi|_{C(A,\iota)}\colon C(A,\iota)\to C(B,\zeta)$.
\end{thm}

\begin{proof}
For the first statement it suffices to show that 
$\wstn$
is full, faithful, and essentially surjective.
To see that it
is full and faithful,
for $C^*$-algebras $A,B$
we must show that
the map
$\phi\mapsto \phi\otimes\id_\KK$ from
$\mor(A,B)$ to $\mor\bigl((A\otimes\KK,1_A\otimes\id_\KK),(B\otimes\KK,1_B\otimes\id_\KK)\bigr)$
is bijective. This is well-known,
although we include an argument for completeness:
let $\psi\colon A\otimes\KK\to M(B\otimes\KK)$ be a nondegenerate homomorphism such that
\[
\psi(1_A\otimes k)=1_B\otimes k\righttext{for all}k\in\KK.
\]
We must show that there exists a unique nondegenerate homomorphism $\phi\colon A\to M(B)$ such that $\psi=\phi\otimes\id_\KK$.
To get the homomorphism $\phi$,
it suffices to show that (the extension to $M(A\otimes\KK)$ of) $\psi$ maps $A\otimes 1_\KK$ into
$M(B\otimes 1_\KK)=M(B)\otimes 1_\KK$:
if $a\in A$ then for all $k\in\KK$ we have
\begin{align*}
(1_B\otimes k)\psi(a\otimes 1)
&=\psi\bigl((1_A\otimes k)(a\otimes 1)\bigr)
\\&=\psi\bigl((a\otimes 1)(1_A\otimes k)\bigr)
\\&=\psi(a\otimes 1)(1_B\otimes k),
\end{align*}
so $\psi(a\otimes 1)\in M(B\otimes 1)$.
Thus there is a unique homomorphism $\phi\colon A\to M(B)$ such that
\[
\phi(a)\otimes 1_\KK=\psi(a\otimes 1_\KK)\righttext{for}a\in A.
\]
To see that 
$\phi$ is nondegenerate,
let $\{a_i\}$ be an approximate identity for $A$.
Since $A\otimes 1$ is nondegenerate in $M(A\otimes\KK)$ and $\psi\colon A\otimes\KK\to M(B\otimes\KK)$ is nondegenerate,
$\psi(a_i\otimes 1)\to 1$ strictly in $M(B\otimes\KK)$.
Let $b\in B$, and take $k\in\KK$ with $\|k\|=1$.
Then
\begin{align*}
\|\phi(a_i)b-b\|
&=\|\phi(a_i)b-b\|\|k\|
\\&=\bigl\|\bigl(\phi(a_i)b-b\bigr)\otimes k\bigr\|
\\&=\|\phi(a_i)b\otimes k-b\otimes k\|
\\&=\|\psi(a_i\otimes 1)(b\otimes k)-b\otimes k\|
\\&\to 0.
\end{align*}
A routine computation now shows that $\psi=\phi\otimes\id_\KK$.

To see that stabilization is essentially surjective,
let $(A,\iota)$ be a $\KK$-algebra.
\propref{AK} gives an isomorphism
\[
\theta\colon C(A,\iota)\otimes\KK\variso A
\]
such that $\theta(a\otimes k)=a\iota(k)$,
and it follows that
\[
\theta\circ (1\otimes\id_\KK)=\iota.
\]
Thus
\[
(C(A,\iota)\otimes\KK,1_A\otimes\id_\KK)\simeq (A,\iota)\righttext{in}\kalgn.
\]
It now follows from generalized abstract nonsense that the assignment $(A,\iota)\mapsto C(A,\iota)$ extends uniquely to a quasi-inverse for the stabilization functor.

Finally, the statement regarding morphisms now follows from the above and the techniques in the proof of
\cite[Section~IV.4, Theorem~1]{maclane}.
\end{proof}

In the above theorem and proof, we did not follow the strategy of nominating a functor as a quasi-inverse and then 
verifying
that the compositions of the two functors in either direction coincide with the appropriate identity functors,
because 
it is more efficient to apply \cite[Section~IV.4, Theorem~1 and its proof]{maclane}. The main benefit of our approach is that
functoriality of the quasi-inverse 
follows automatically from the rest, by Mac Lane's theorem and its proof.

\begin{defn}
The \emph{nondegenerate destabilization functor} $\dstn\colon \kalgn\to\csn$ is
the quasi-inverse of $\wstn$ with object map
$(A,\iota)\mapsto C(A,\iota)$,
described in the above theorem,
and for a morphism $\psi\colon (A,\iota)\to (B,\zeta)$ in $\kalgn$ we write $C(\psi)\colon C(A,\iota)\to C(B,\zeta)$ for the destabilization.
\end{defn}

\thmref{destabilize} shows that 
destabilization gives rise to
an inversion 
of the stabilization functor $B\mapsto B\otimes \KK$,
with forgetful functor $(A,\iota)\mapsto A$.
It is easy to check that this inversion is good:
the objects in the image of the forgetful functor are precisely those $C^*$-algebras that are isomorphic in $\csn$ to objects in the image of $\stn$,
and
if $(A,\iota)$ is an object in $\kalgn$ and $\theta\colon A\variso B$ is an isomorphism in $\csn$,
then $\zeta:=\theta\circ\iota\colon \KK\to M(B)$ gives
a $\KK$-algebra $(B,\zeta)$
such that $\theta\colon (A,\iota)\to (B,\zeta)$ is the unique isomorphism in $\kalgn$
with domain $(A,\iota)$ whose image under the forgetful functor is the given isomorphism $\theta\colon A\to B$.

Since we have a good inversion, \cite[Proposition~4.2]{koqlandstad} gives the following immediate consequence:

\begin{cor}
For $j=1,2$ let $\iota_j\colon \KK\to M(A)$ be a nondegenerate homomorphism,
and suppose that $B_j$ is a $C^*$-algebra isomorphic to the relative commutant $C(A,\iota_j)$.
Then $B_1\simeq B_2$ if and only if there is an automorphism $\alpha$ of $A$ such that
\[
\alpha\circ\iota_1=\iota_2.
\]
\end{cor}

\section{Extending to Hilbert bimodules}\label{bimodule}

\begin{prop}\label{XK}
Let $(A,\iota)$ and $(B,\zeta)$ be $\KK$-algebras
let $X$ be an $A-B$ Hilbert bimodule.
Let $C(A,\iota)$ and $C(B,\zeta)$ be the associated relative commutants of $\KK$,
and put
\begin{equation}
C(X,\iota,\zeta)=\{x\in M(X):\iota(k)\cdot x=x\cdot \zeta(k)\in X\text{ for all }k\in\KK\}.
\end{equation}
Then:
\begin{enumerate}
\item
$C(X,\iota,\zeta)$ becomes an $C(A,\iota)-C(B,\zeta)$ Hilbert bimodule with operations inherited from the $M(A)-M(B)$ Hilbert bimodule $M(X)$,
and moreover has linking algebra
\[
C(L,\omega)=L(C(X,\iota,\zeta)),
\]
where $\omega=\smtx{\iota&0\\0&\zeta}$.

\item
When $X$ is regarded as a $(C(A,\iota)\otimes\KK)-(C(B,\zeta)\otimes\KK)$ Hilbert bimodule via the isomorphisms
$A\simeq C(A,\iota)\otimes\KK$
and
$B\simeq C(B,\zeta)\otimes\KK$
of \propref{AK},
there is an isomorphism
\[
C(X,\iota,\zeta)\otimes\KK\variso X
\]
of $(C(A,\iota)\otimes\KK)-(C(B,\zeta)\otimes\KK)$ Hilbert bimodules
given on elementary tensors by
\[
x\otimes k\mapsto x\cdot \zeta(k).
\]

\item
The $C(A,\iota)-C(B,\zeta)$ Hilbert bimodule $C(X,\iota,\zeta)$ is right-full \(respectively, left-full\) if and only if the $A-B$ Hilbert bimodule $X$ is; in particular, $C(X,\iota,\zeta)$ is an imprimitivity bimodule if and only if $X$ is.
\end{enumerate}

Furthermore, if $C$ and $D$ are $C^*$-algebras and $Y$ is a $C-D$ Hilbert bimodule, then
\[
C(Y\otimes\KK,1_C\otimes\id,1_D\otimes\id)\simeq Y
\]
as $C-D$ Hilbert bimodules.
\end{prop}

Note that in the final statement of the above proposition we turned the $C(C\otimes\KK,1_C\otimes\id)-C(D\otimes\KK,1_D\otimes\id)$ Hilbert bimodule $C(Y\otimes\KK,1_C\otimes\id,1_D\otimes\id)$ into a $C-D$ Hilbert bimodule
via the isomorphisms $C(C\otimes\KK,1_C\otimes\id)\simeq C$ and $C(D\otimes\KK,1_D\otimes\id)\simeq D$.

\begin{proof}
(1)
This could be proven by adapting 
the arguments in the proof of \propref{AK}
--- for example, we would start with $X_{ij}=e_{ii}\cdot X\cdot f_{jj}$,
where $\{e_{ij}\}$ and $\{f_{ij}\}$ are matrix units associated to $\iota$ and $\zeta$, respectively.
But it is instructive to instead \emph{apply} 
\propref{AK}
to the linking algebra.
Let $L=L(X)=\smtx{A&X\\\wilde X&B}$ be the linking algebra of the Hilbert bimodule $X$,
with associated complementary  projections
\[
p=\mtx{1_A&0\\0&0}\midtext{and}q=\mtx{0&0\\0&1_B}
\]
in $M(L)$.
Then
\[
\omega(k)=\mtx{\iota(k)&0\\0&\zeta(k)}
\]
gives a nondegenerate homomorphism $\omega\colon \KK\to M(L)$,
and
\[
g_{ij}:=\mtx{e_{ij}&0\\0&f_{ij}}
\]
is an associated system
of matrix units.
Thus, by \propref{AK}
there is an isomorphism
\[
\theta\colon C(L,\omega)\otimes\KK\variso L
\]
given on elementary tensors by
\[
\theta(c\otimes k)=c\omega(k).
\]

We easily see that $p\in M(C(L,\omega))$:
for all $i,j$,
\begin{align*}
g_{ij}p
&=\mtx{e_{ij}&0\\0&f_{ij}}\mtx{1&0\\0&0}
=\mtx{e_{ij}&0\\0&0}
=\mtx{1&0\\0&0}\mtx{e_{ij}&0\\0&f_{ij}}
=pg_{ij}.
\end{align*}
Thus $p$ and $q=1-p$ are complementary  projections in $M(C(L,\omega))$,
so
$pC(L,\omega) q$ is a $pC(L,\omega) p-qC(L,\omega) q$ Hilbert bimodule.

Claim:
\[
C(L,\omega)=\mtx{C(A,\iota)&C(X,\iota,\zeta)\\{*}&C(B,\zeta)},
\]
where we no longer bother to specify the lower left corners of linking algebras.
By, for example, \cite[Proposition~1.51]{enchilada},
\[
M(L)=\mtx{M(A)&M(X)\\{*}&M(B)}.
\]
Thus
$pM(L)p=M(A)$.
Consequently, 
\[
pC(L,\omega) p=\{a\in M (A):g_{ij}\smtx{a&0\\0&0}=\smtx{a&0\\0&0}g_{ij}\in L\text{ for all }i,j\}.
\]
Since 
\[
g_{ij}\mtx{a&0\\0&0}=\mtx{e_{ij}&0\\0&f_{ij}}\mtx{a&0\\0&0}=\mtx{e_{ij}a&0\\0&0},
\]
and similarly $\smtx{a&0\\0&0}g_{ij}=\smtx{ae_{ij}&0\\0&0}$,
and since $L\cap M(A)=A$,
we see that $pC(L,\omega) p=C(A,\iota)$.
Similar computations 
show that $qC(L,\omega) q=C(B,\zeta)$ and $pC(L,\omega) q=C(X,\iota,\zeta)$,
proving the claim.
It follows that
$C(X,\iota,\zeta)$ is a $C(A,\iota)-C(B,\zeta)$ Hilbert bimodule,
with linking algebra
\[
C(L,\omega)=L(C(X,\iota,\zeta)).
\]

(2)
Note that the linking algebra of the $(C(A,\iota)\otimes\KK)-(C(B,\zeta)\otimes\KK)$ Hilbert bimodule $C(X,\iota,\zeta)\otimes\KK$ is
\begin{align*}
L(C(X,\iota,\zeta)\otimes\KK)
&=\mtx{C(A,\iota)\otimes\KK&C(X,\iota,\zeta)\otimes\KK\\{*}&C(B,\zeta)\otimes\KK}
\\&\simeq L(C(X,\iota,\zeta))\otimes\KK
=C(L,\omega)\otimes\KK
\end{align*}
(see, e.g., \cite[Remark~1.50]{enchilada}).
The isomorphism $\theta\colon C(L,\omega)\otimes\KK\variso L$ 
of \propref{AK}
restricts on the upper right corner to an isomorphism of
$(C(A,\iota)\otimes\KK)-(C(B,\zeta)\otimes\KK)$ Hilbert bimodules.
The upper right corner of $L(C(X,\iota,\zeta)\otimes\KK)$ is
$C(X,\iota,\zeta)\otimes\KK$,
and the upper right corner of $L$ is
$X$.
For all $x\in C(X,\iota,\zeta)=pC(L,\omega) q$
we have
\[
x\otimes g_{ij}=x\otimes f_{ij}
\midtext{and}xg_{ij}=x\cdot f_{ij}.
\]
Thus $\theta$ restricts on the upper right corner to an isomorphism $C(X,\iota,\zeta)\otimes\KK\simeq X$
such that $x\otimes f_{ij}\mapsto x\cdot \zeta(f_{ij})$,
and (2) follows.

(3)
We prove that if $X$ is right-full then so is $C(X,\iota,\zeta)$; the proof on the other side is similar, and we omit it.
By
\cite[Proposition~1.48]{enchilada},
since $X$ is right-full the projection $p$ is full in $M(L)$,
and
it suffices to show that $p$ is full in $M(C(L,\omega))$.

Note that the canonical extension to multiplier algebras of the isomorphism $\theta\colon C(L,\omega)\otimes\KK\to L$ takes the same form on elementary tensors:
$\theta(c\otimes k)=c\omega(k)$ for all $c\in M(C(L,\omega)),k\in M(\KK)$.
Thus
\[
\theta(p\otimes 1_\KK)=p,
\]
so
$\theta\inv$ takes the  projection $p$ in $M(L)$ to $p\otimes 1$.
Thus $p\otimes 1$ is full in $M(C(L,\omega)\otimes\KK)$, and
it follows that $p$ is a full projection in $M(C(L,\omega))$.

Finally, let $Y$ be a $C-D$ Hilbert bimodule.
Then from the above we have a $C(C\otimes\KK,1_C\otimes\id)-C(D\otimes\KK,1_D\otimes\id)$ Hilbert bimodule $C(Y\otimes\KK,1_C\otimes\id,1_D\otimes\id)$,
which we regard as a $C-D$ Hilbert bimodule via the isomorphisms $C(C\otimes\KK,1_C\otimes\id)\simeq C$ and $C(D\otimes\KK,1_D\otimes\id)\simeq D$.
It follows that we have isomorphisms
\begin{align*}
&\mtx{C&C(Y\otimes\KK,1_C\otimes\id,1_D\otimes\id)\\{*}&D}
\\&\quad\simeq\mtx{C(C\otimes\KK,1_C\otimes\id)&C(Y\otimes\KK,1_C\otimes\id,1_D\otimes\id)\\{*}&C(D\otimes\KK,1_D\otimes\id)}
\\&\quad=L\bigl(C(Y\otimes\KK,1_C\otimes\id,1_D\otimes\id)\bigr)
\\&\quad=C\left(L(Y\otimes\KK),\mtx{1_C\otimes\id&0\\0&1_D\otimes\id}\right)
\righttext{(by (1))}
\\&\quad\simeq C\left(L(Y)\otimes\KK,1_{L(Y)}\otimes\id\right)
\\&\quad\simeq L(Y)
\\&\quad=\mtx{C&Y\\{*}&D}
\end{align*}
of $C^*$-algebras that preserve the projection $\smtx{1_C&0\\0&0}$,
and it follows that
$C(Y\otimes\KK,1_C\otimes\id,1_D\otimes\id)\simeq Y$ as $C-D$ Hilbert bimodules.
\end{proof}

\section{Enchilada category equivalence}\label{enchilada}

We now want an analogue of the equivalence of \thmref{destabilize}, but for the enchilada categories.

\begin{defn}
The \emph{enchilada category $\kalge$ of  $\KK$-algebras}
has the same objects as $\kalgn$,
but now when we say $[X]\colon (A,\iota)\to (B,\zeta)$ is a morphism in the category we mean 
$[X]\colon A\to B$ in $\cse$.
\end{defn}

\begin{rem}
In the spirit of \cite{hkrwnatural}, we could call $\kalge$ a ``semi-comma category'' --- it is not quite the comma category of objects of $\cse$ under $\KK$, because we don't require the morphisms to be in any way compatible with the embeddings of $\KK$.
\end{rem}

If $X$ is a nondegenerate $A-B$ correspondence,
it is well-known that the external tensor product
$X\otimes\KK$ is a nondegenerate $(A\otimes\KK)-(B\otimes\KK)$-correspondence.
If $\phi\colon X\to Y$ is an isomorphism of (nondegenerate) $A-B$ correspondences,
then $\phi\otimes\id_\KK\colon X\otimes\KK\to Y\otimes\KK$ is an isomorphism of $(A\otimes\KK)-(B\otimes\KK)$ correspondences.
By
\cite[Lemma~2.12]{enchilada},
if we have an
$A-B$ correspondence $X$ and a 
$B-C$ correspondence $Y$ (both nondegenerate),
then there is an $(A\otimes\KK)-(C\otimes\KK)$ correspondence isomorphism
\[
\Theta\colon (X\otimes\KK)\otimes_{B\otimes\KK}(Y\otimes\KK)\variso (X\otimes_BY)\otimes\KK
\]
such that
\[
\Theta\bigl((x\otimes k)\otimes (y\otimes m)\bigr)=(x\otimes y)\otimes km
\righttext{for}x\in X,y\in Y,k,m\in\KK.
\]
This justifies the following:

\begin{defn}
The \emph{enchilada stabilization functor} $\ste\colon \cse\to \cse$
is the same as $\stn$ on objects,
but is given on morphisms by
$[X]\mapsto [X\otimes\KK]$,
and the
\emph{embellished enchilada stabilization functor} $\wste\colon \cse\to \kalge$
is the same as $\wstn$ on objects,
and is the same as $\ste$ on morphisms.
\end{defn}

\begin{thm}\label{destabilize enchilada}
The embellished enchilada stabilization functor $\wste\colon \cse\to\kalge$
is a category equivalence,
and has a 
unique 
quasi-inverse functor
whose object map is $(A,\iota)\mapsto C(A,\iota)$.
\end{thm}

\begin{proof}
We follow the same overall strategy as in \thmref{destabilize}:
we show that 
$\wste$
is full and faithful,
and that for every $\KK$-algebra $(A,\iota)$ we have
\[
(C(A,\iota)\otimes\KK,1_A\otimes\id_\KK)\simeq (A,\iota)
\righttext{in}\kalge.
\]
This latter part, a particular way for stabilization to be essentially surjective that by generalized abstract nonsense will also take care of the second statement of the theorem,
follows immediately from its earlier counterpart \propref{AK},
because isomorphism of objects in the nondegenerate category is stronger than isomorphism in the enchilada category.

So, it remains to show that
the map
\begin{equation}\label{1-1 onto}
[X]\mapsto [X\otimes\KK]
\end{equation}
from
$\mor(A,B)$
to
$\mor\bigl((A\otimes\KK,1_A\otimes\id_\KK),(B\otimes\KK,1_B\otimes\id_\KK)\bigr)$
is bijective.

For the surjectivity of \eqref{1-1 onto},
we argue slightly more generally:
let $(A,\iota)$ and $(B,\zeta)$ be $\KK$-algebras and let $X$ be a nondegenerate $A-B$ correspondence.
We will show that there is a nondegenerate $C(A,\iota)-C(B,\zeta)$ correspondence $C(X,\iota,\zeta)$ such that
\[
C(X,\iota,\zeta)\otimes\KK\simeq X
\]
as $A-B$ correspondences
(where the $(C(A,\iota)\otimes\KK)-(C(B,\zeta)\otimes\KK)$ correspondence $C(X,\iota,\zeta)\otimes\KK$ is turned into an $A-B$ correspondence with the help of the isomorphisms
$C(A,\iota)\otimes\KK\simeq A$ and $C(B,\zeta)\otimes\KK\simeq B$).
To apply this to the surjectivity question,
given a $(C\otimes\KK)-(D\otimes\KK)$ correspondence $X$,
find a $C(C\otimes\KK,\iota)-C(D\otimes\KK,\zeta)$ correspondence $C(X,\iota,\zeta)$ such that
$C(X,\iota,\zeta)\otimes\KK\simeq X$ as
$(C\otimes\KK)-(D\otimes\KK)$ correspondences,
then regard $C(X,\iota,\zeta)$ as a $C-D$ correspondence via the isomorphisms
$C(C\otimes\KK,\iota)\simeq C$ and $C(D\otimes\KK,\zeta)\simeq D$.

Now,
$X$ can also be regarded as a left-full $E-B$ Hilbert bimodule, where $E=\KK(X)$.
The left-module multiplication of $A$ on $X$ is given by a nondegenerate homomorphism
\[
\phi\colon A\to \LL(X),
\]
and then the composition
\[
\iota_E:=\phi\circ\iota\colon \KK\to M(E)
\]
is nondegenerate, giving $E$ the structure of a $\KK$-algebra.
Let $C(X,\iota_E,\zeta)$ be the left-full $C(E,\iota_E)-C(B,\zeta)$ Hilbert bimodule from \propref{XK}.
Letting
$C(\phi)\colon C(A,\iota)\to M(C(E,\iota_E))$ be the nondegenerate homomorphism 
given by the nondegenerate destabilization functor $\dstn$,
$C(X,\iota_E,\zeta)$ can be regarded as a nondegenerate $C(A,\iota)-C(B,\zeta)$ correspondence,
and we shorten the notation to $C(X,\iota,\zeta)$.
From \propref{XK} we have an isomorphism
\[
\theta_X\colon C(X,\iota,\zeta)\otimes\KK\variso X
\]
of $E-B$ Hilbert bimodules,
and we need to show that $\theta_X$ preserves the left $A$-module structures,
so that we will have $[C(X,\iota,\zeta)\otimes\KK]=[X]$ as morphisms from $A$ to $B$ in $\kalge$.
Due to the isomorphism $\theta_A\colon C(A,\iota)\otimes\KK\variso A$ of \propref{AK},
the following computation suffices:
for $a\in C(A,\iota)$, $x\in X$, and $k,m\in\KK$ we have
\begin{align*}
\theta_X\bigl(\theta_A(a\otimes k)\cdot (x\otimes m)\bigr)
&=\theta_X\bigl((a\otimes k)\cdot (x\otimes m)\bigr)
\\&=\theta_X(a\cdot x\otimes km)
\\&=(a\cdot x)\cdot \zeta(km)
\\&=a\cdot x\cdot \zeta(k)\zeta(m)
\\&=a\cdot \iota_E(k) x\cdot \zeta(m)
\\&=a\cdot \iota(k)\cdot x\cdot \zeta(m)
\\&=\bigl(a\iota(k)\bigr)\cdot x\cdot \zeta(m)
\\&=\theta_A(a\otimes k)\cdot \theta_X(x\otimes m).
\end{align*}
Note that the definition of the relative-commutant Hilbert bimodule from \propref{XK} is
\[
C(X,\iota,\zeta)=\{x\in M(X):\iota_E(k)x=x\cdot \zeta(k)\in X\text{ for all }k\in\KK\},
\]
but this can be computed without appealing to the imprimitivity algebra $E$:
it follows from the construction that
\[
C(X,\iota,\zeta)=\{x\in M(X):\iota(k)\cdot x=x\cdot \zeta(k)\in X\text{ for all }k\in\KK\}.
\]

Finally, for the injectivity of \eqref{1-1 onto},
it suffices to show that if $[X]$ is a morphism from $A$ to $B$ in $\cse$,
and if we 
let $\iota=1_A\otimes\id_\KK$ and $\zeta=1_B\otimes\id_\KK$,
and
regard $C(X\otimes\KK,\iota,\zeta)$ as an $A-B$ correspondence
via the isomorphisms $C(A\otimes\KK,\iota)\simeq A$ and $C(B\otimes\KK,\zeta)\simeq B$,
then $[X]=[C(X\otimes\KK,\iota,\zeta)]$ in $\cse$,
i.e.,
\[
C(X\otimes\KK,\iota,\zeta)\simeq X\righttext{as $A-B$ correspondences}.
\]
Letting $E=\KK(X)$,
we know from \propref{XK} that there is an isomorphism
\[
\sigma\colon C(X\otimes\KK,\iota,\zeta)\variso X
\]
of $E-B$ Hilbert bimodules, and we need to know that it preserves the left $A$-module structures.
It will help to be a little more precise concerning $\sigma$:
the last part of the proof of \propref{XK} gives an isomorphism
\[
\psi=\mtx{\id_E&\sigma\\{*}&\id_B}\colon \mtx{E&C(X\otimes\KK,\iota,\zeta)\\{*}&B}\to \mtx{E&X\\{*}&B}
\]
of matrix algebras.
Let $a\in A$ and $y\in C(X\otimes\KK,\iota,\zeta)$,
and let $\{e_i\}$ be an approximate identity for $E$. Then
\begin{align*}
\mtx{0&\sigma(a\cdot y)\\0&0}
&=\lim \mtx{0&\sigma\bigl(\phi_{C(X\otimes\KK,\iota,\zeta)}(a)e_iy\bigr)\\0&0}
\\&=\lim \mtx{0&\phi_X(a)e_i\sigma(y)\\0&0}
\\&=\lim \mtx{\phi_X(a)e_i&0\\0&0}\mtx{0&\sigma(y)\\0&0}
\\&=\mtx{\phi_X(a)&0\\0&0}\mtx{0&\sigma(y)\\0&0},
\intertext{by the strict-continuity clause of \cite[Proposition~1.51]{enchilada},}
\\&=\mtx{0&\phi_X(a)\sigma(y)\\0&0}
\\&=\mtx{0&a\cdot \sigma(y)\\0&0},
\end{align*}
so
\[
\sigma(a\cdot y)=a\cdot \sigma(y).
\qedhere
\]
\end{proof}

\begin{rem}
\thmref{destabilize enchilada} gives an inversion (in the technical sense of \cite[Definition~4.1]{koqlandstad}) of the enchilada stabilization functor
$\ste$
However, this inversion is \emph{not} good.
To see why, just note that
a $C^*$-algebra 
can be isomorphic
in the category $\cse$ to something in the image of the forgetful functor
$(A,\iota)\mapsto A$
without itself being in this image,
since nonstable $C^*$-algebras can be Morita equivalent to stable ones,
and
so
the image of the forgetful functor $\kalge\to \cse$ is not contained in the essential image of $\ste\colon \cse\to \cse$.
However, similarly to \cite[Remark~5.6]{koqlandstad}, the forgetful functor is conservative (and hence so is $\stn$) and essentially surjective.
\end{rem}

\begin{rem}
Perhaps it is of interest to note that our use of category-theory techniques obviated the need to directly establish that 
the assignments $X\mapsto C(X,\iota,\zeta)$ from $\KK$-compatible $A-B$ correspondences to $C(A,\iota)-C(B,\zeta)$ correspondences
is functorial up to isomorphism;
this would have required that we prove an isomorphism of the form
\[
C(X\otimes_B Y,\iota)\simeq C(X,\iota,\zeta)\otimes_{C(B,\zeta)} C(Y,\iota),
\]
but in fact it follows from the properties of category equivalences. 
\end{rem}

\subsection{Restricting to stable algebras}
Recall that isomorphism in the nondegenerate category is the usual isomorphism of $C^*$-algebras,
whereas isomorphism in the enchilada category is Morita equivalence of $C^*$-algebras.
By \cite[Theorem~3.4]{bgr}, every imprimitivity bimodule between stable $C^*$-algebras possessing strictly positive elements (in particular, between separable stable algebras) is isomorphic to one coming from an isomorphism between the $C^*$-algebras.
Since we are considering the stabilization process, it is natural to ponder the ramifications for our two categories.

A morphism in $\csn$
is a nondegenerate homomorphism $\phi\colon A\to M(B)$,
and this determines a \emph{standard} $A-B$ correspondence $B_\phi$,
where the Hilbert $B$-module structure comes from the algebraic operations of $B$
and the left $A$-module structure comes from $\phi$.
Two standard $A-B$ correspondences $B_\phi$ and $B_\psi$ are isomorphic if and only if there is a unitary $u\in M(B)$ such that $\psi=\ad u\circ\phi$
\cite[Proposition~2.3]{taco}
(see also \cite[Corollary~3.2]{bgr}).
It follows that there is a functor $E\colon \csn\to\cse$ that is the identity on objects and takes $\phi$ to $[B_\phi]$.
This functor $E$ is 
trivially essentially surjective, but it is
not faithful, and is certainly not full, since nonisomorphic $C^*$-algebras can be Morita equivalent.

We wish to make the point here that $E$ becomes full (in a nontrivial way) if we restrict to a suitable subcategory.
In view of the Brown-Green-Rieffel theorem,
it seems clear that we should restrict to stable $C^*$-algebras possessing strictly positive elements.
Our strategy is to ``factor'' a correspondence into a homomorphism and an imprimitivity bimodule.
Given a nondegenerate $A-B$ correspondence $X$, we have a nondegenerate homomorphism $\varphi\colon A\to \LL(X)=M(\KK(X))$,
and $X$ is also a left-full $\KK(X)-B$ Hilbert bimodule.
We actually want an imprimitivity bimodule, so we should restrict to full correspondences.

But there is a subtlety: although we've already agreed that we'll require $A$ and $B$ to have strictly positive elements, there's no reason to expect that $\KK(X)$ will have any --- for example, we could have $A=B=\C$, and $X$ could be a nonseparable Hilbert space, and then $\KK(X)$ does not have a strictly positive element (essentially because every compact operator has separable range).
In fact, assuming $B$ has a strictly positive element, we do not see any obvious condition on a Hilbert $B$-module $X$ that would guarantee that $\KK(X)$ has a strictly positive element too.
To simplify the discussion, therefore, we stick to separable $C^*$-algebras and correspondences. This removes the preceding worry, since if $X$ is a separable Hilbert $B$-module and $B$ is separable, then $\KK(X)$ is separable also.

So, we restrict to the full subcategory $\sscs$
of $\csn$ whose objects are the separable stable $C^*$-algebras.
The functor $E$ takes $\sscs$ into the (nonfull) subcategory $\ssfcsen$ of $\cse$ whose objects are again separable and stable, 
and whose morphisms are isomorphism class of \emph{full} nondegenerate separable $C^*$-correspondences.
It is an easy consequence of the Brown-Green-Rieffel theorem that in fact $E$ maps $\sscs$ \emph{onto} $\ssfcsen$.
To see this, let $X$ be a full nondegenerate separable $A-B$ correspondence.
Then $\KK(X)$ is separable, and
$X$ is a $\KK(X)-B$ imprimitivity bimodule,
so by \cite[Theorem~3.4]{bgr} there exist isomorphisms
\[
\theta\colon \KK(X)\variso B
\]
and
\[
\Theta\colon X\variso B_\theta.
\]
Letting $\varphi\colon A\to M(\KK(X))=\LL(X)$ be the left $A$-module structure,
we get a full nondegenerate $A-B$ correspondence $B_{\theta\circ\varphi}$.
Obviously $\Theta$ preserves the left $A$-module actions,
so gives an isomorphism of $A-B$ correspondences.
Thus
\[
[X]=E\bigl(B_{\theta\circ\varphi}\bigr)\righttext{in $\ssfcsen$}.
\]

\subsection{Double crossed products}
There is a connection to crossed-product duality:
for an
infinite
second-countable
locally compact group $G$, the result
\cite[Proposition~5.3]{cldx}, translated into the context of full-crossed-product duality, can be restated as follows: the assignments
\begin{equation}\label{double}
(A,\alpha)\mapsto (A\rtimes_\alpha G\rtimes_{\what\alpha} G,\what{\what\alpha},i_G\rtimes G)
\midtext{and}
\phi\mapsto \phi\rtimes G\rtimes G
\end{equation}
give an equivalence of the nondegenerate category of $G$-actions with an equivariant category of $\KK$-algebras.
Here
\[
i_G\rtimes G\colon C^*(G)\rtimes_{\delta_G} G\to M(A\rtimes_\alpha G\rtimes_{\what\alpha} G)
\]
is the crossed product of the equivariant homomorphism $i_G\colon C^*(G)\to M(A\rtimes_\alpha G)$.
The functor \eqref{double} can be regarded as the composition of two ``crossed-product-type'' functors:
\begin{enumerate}
\item
$(A,\alpha)\mapsto (A\rtimes_\alpha G,\what\alpha,i_G)$
from the nondegenerate category actions to a nondegenerate category of ``equivariant'' coactions,
where a typical object  $(B,\delta,V)$ comprises a $\delta_G-\delta$ equivariant nondegenerate homomorphism $V\colon C^*(G)\to M(B)$ (see \cite[Section~2]{koqlandstad}, and

\item
$(B,\delta,V)\mapsto (B\rtimes_\delta G,\what\delta,V\rtimes G)$
from the above category of equivariant coactions to the equivariant category of $\KK$-algebras.
\end{enumerate}
Anyhow, using the canonical isomorphism
\[
(C^*(G)\rtimes_{\delta_G} G,\what{\delta_G})\simeq (\KK,\ad\rho),
\]
where $\KK=\KK(L^2(G))$ and $\rho$ is the right regular representation of $G$ on $L^2(G)$,
the homomorphism
\[
V\rtimes G\colon C^*(G)\rtimes_{\delta_G} G\to M(B\rtimes_\delta G)
\]
can be identified with an $\ad\rho-\what\delta$ equivariant homomorphism $\KK\to M(B\rtimes_\delta G)$,
and this is how the crossed-product-type functor in (2) above can be regarded as going from equivariant coactions to equivariant $\KK$-algebras.

On the other hand, the assignments
\[
(A,\alpha)\mapsto (A\otimes\KK,\alpha\otimes\ad\rho,1_{M(A)}\otimes\id_\KK)
\midtext{and}
\phi\mapsto \phi\otimes\id_\KK
\]
give another equivalence of the nondegenerate category of actions to the equivariant category of $\KK$-algebras.
The canonical map
\[
\Phi\colon A\rtimes_\alpha G\rtimes_{\what\alpha} G\to A\otimes\KK
\]
gives a natural isomorphism between these two equivalences.

In \cite[Proposition~5.3]{cldx} there is no specified quasi-inverse,
but the techniques of this paper show that
we can recover $A$ up to isomorphism from the double crossed product $A\rtimes_\alpha G\rtimes_{\what\alpha} G$ as the relative commutant
\[
C(A\rtimes_\alpha G\rtimes_{\what\alpha} G,i_G\rtimes G).
\]
In fact, this is how Fischer constructs the maximalization of a coaction in \cite{fischer}.
Of course, the corresponding isomorphic copy of the given action $\alpha$ is then recovered by restricting to the relative commutant of (the canonical extension to multipliers of) the inner action
$\ad (j_{A\rtimes_\alpha G}\circ i_G)$.

A similar discussion can be given for the dual situation, starting with a maximal coaction $(A,\delta)$ and producing the equivariant $\KK$-algebra
\[
(A\rtimes_\delta G\rtimes_{\what\delta} G,\what{\what\delta},j_G\rtimes G).
\]


\providecommand{\bysame}{\leavevmode\hbox to3em{\hrulefill}\thinspace}
\providecommand{\MR}{\relax\ifhmode\unskip\space\fi MR }
\providecommand{\MRhref}[2]{%
  \href{http://www.ams.org/mathscinet-getitem?mr=#1}{#2}
}
\providecommand{\href}[2]{#2}

\end{document}